\documentclass[a4paper,12pt]{article}
\UseRawInputEncoding
\usepackage{amsmath}
\usepackage{ascmac}
\usepackage{amssymb}
\usepackage{multicol}
\usepackage{graphicx}
\usepackage[dvipdfmx]{color}
\usepackage{amscd}
\usepackage{url}
\usepackage{comment}
\usepackage{bm}
\usepackage{theorem}
\usepackage{latexsym}
\usepackage{mathrsfs}
\usepackage{framed,color}
\usepackage{fancybox}
\usepackage{ulem}
\usepackage[dvipdfmx]{hyperref}
\usepackage[all]{xy}
\theorembodyfont{\normalfont}
\newtheorem{theorem}{Theorem}[section]
\newtheorem{definition}[theorem]{Definition}
\newtheorem{proposition}[theorem]{Proposition}
\newtheorem{lemma}[theorem]{Lemma}
\newtheorem{corollary}[theorem]{Corollary}
\newtheorem{remark}[theorem]{Remark}
\newtheorem{example}[theorem]{Example}

\definecolor{shadecolor}{gray}{0.85}
\makeatletter
\newenvironment{proof}[1][\proofname]{\par
  \normalfont
  \topsep6\p@\@plus6\p@ \trivlist
  \item[\hskip\labelsep{\bfseries #1}\@addpunct{\bfseries}]\ignorespaces
}{
  \endtrivlist
}
\newcommand{\proofname}{\textit{Proof}.}

\newcommand{\ctext}[1]{\raise0.2ex\hbox{\textcircled{\scriptsize{#1}}}}

	\@addtoreset{equation}{section}

	\@addtoreset{figure}{section}

	\@addtoreset{table}{section}
	
\makeatother
\def\qed{\hfill$\Box$}
\title{Gauss-Bonnet type formulas for frontal bundles over surfaces with boundary and their applications}
\author{Kyoya Hashibori\footnote{hashibori.kyoya.a7@elms.hokudai.ac.jp}}
\date{Department of Mathematics, Graduate School of Science, Hokkaido University, Kita-10 Nishi-8, Kita-ku, Sapporo 060-0810, JAPAN}
\begin{document}

\maketitle

\abstract

We define a frontal bundle by imposing a compatibility condition on two types of coherent tangent bundles over a surface with boundary. Since it is known that there are two Gauss-Bonnet type formulas for coherent tangent bundles, we obtain four Gauss-Bonnet type formulas for frontal bundles. Furthermore, since two coherent tangent bundles over the surface are related to each other by the compatibility condition, an appropriate combination of these Gauss-Bonnet type formulas leads to new formulas that relate a frontal to its unit normal vector field using the Gaussian curvature, the Euler characteristic, etc. If the extrinsic curvature of a frontal is non-zero bounded, these formulas lead to formulas that relate the number of singular points to the Euler characteristic.

\section{Introduction}\label{sec1}

The well-known classical Gauss-Bonnet theorem is the following formula that explicitly describes the relationship between the Gaussian curvature $K$ of a regular surface $M$ and the Euler characteristic $\chi(M)$:
\begin{equation}
\int_MKdA=2\pi\chi(M).\label{1.1}
\end{equation}

Saji, Umehara, and Yamada generalized $(\ref{1.1})$ to frontals, which are surfaces that admit singular points, using the theory of coherent tangent bundles, where a coherent tangent bundle is a natural intrinsic formulation of frontals (see \cite{1,2,3,4,5,6}). Later, Domitrz and Zwierzy\'{n}ski generalized the theory of coherent tangent bundles to the case of surfaces with boundary, and they generalized the Gauss-Bonnet type formulas derived by Saji, Umehara, and Yamada to the case of surfaces with boundary (see \cite{1}).

Saji, Umehara, and Yamada defined a notion of a frontal bundle in \cite{2,4}. A frontal bundle is a structure obtained by combining two types of coherent tangent bundles over a surface under a compatibility condition. For example, we can construct a frontal bundle from a frontal and its unit normal vector field (see Example $\ref{ex3.2}$). In \cite{4}, four Gauss-Bonnet type formulas for frontal bundles over surfaces without boundary are derived, and by combining these formulas, several formulas are derived. For example, the Bleecker-Wilson formula relating the number of cusps of a Gauss map on a regular surface to the Euler characteristic of a subsurface such that the Gaussian curvature is negative is easily derived by a combination of Gauss-Bonnet type formulas (see \cite[Theorem $3.12$]{4}). Furthermore, applications are given for the case where the extrinsic curvature defined on the set of regular points of a frontal is non-zero bounded, and it is shown that the number of singular points and the Euler characteristic are affected by the positivity or negativity of the extrinsic curvature (see \cite[Theorem $3.28$, $3.30$]{4})．

Based on these previous studies, this paper generalizes the theory of frontal bundles to the case of surfaces with boundary, and presents formulas for the relationship between frontals and their unit normal vector fields. The following assertion is the main result.

\begin{theorem}
\label{thm1.1}

{\it Let $M$ be a compact oriented surface with boundary, $(N,g)$ a $3$ dimensional Riemannian manifold, $f:M\to N$ a co-orientable frontal, and $\nu$ a unit normal vector field of $f$. We suppose that $f$ and $\nu$ allow only singular points of the first kind and admissible singular points of the second kind, and that the set of singular points $\Sigma$ of $f$ and the set of singular points $\Sigma_\star$ of $\nu$ are transversal to the boundary $\partial{M}$. (Here, a singular point of $f$ (resp. $\nu$) is a point where the rank of the derivative $df$ of $f$ (resp. the covariant derivative $D\nu$ of $\nu$) is smaller than $2$ (see Example $\ref{ex3.2}$).) Then, we have the following four formulas:
\begin{itemize}
\item[$(1)$]$\displaystyle{4\pi\chi(M^-_\star)=2\int_{\Sigma}\kappa_sds+\int_{\partial{M}}\kappa_gds}$

\ \ \ $\displaystyle{-2\int_{M^-}Kd\widehat{A}-\sum_{p\in(\Sigma\cap\partial{M})^{\mathrm{null}}}\left(2\alpha^+(p)-\pi\right)}$

\ \ \ $\displaystyle{-\left(\int_{\partial{M}\cap M^+_\star}\kappa_g^\star ds_\star-\int_{\partial{M}\cap M^-_\star}\kappa_g^\star ds_\star\right)+2\pi\left(\#S^+_\star-\#S^-_\star\right)}$

\ \ \ $\displaystyle{-\pi\left(\#(\Sigma_\star\cap\partial{M})^{\mathrm{null}}+2\#(\Sigma_\star\cap\partial{M})^-\right)}$.
\item[$(2)$]$\displaystyle{4\pi\chi(M^-)=2\int_{\Sigma_\star}\kappa_s^\star ds_\star+\int_{\partial{M}}\kappa_g^\star ds_\star}$

\ \ \ $\displaystyle{-2\int_{M^-_\star}K_\star d\widehat{A}_\star-\sum_{p\in(\Sigma_\star\cap\partial{M})^{\mathrm{null}}}\left(2\alpha^+_\star(p)-\pi\right)}$

\ \ \ $\displaystyle{-\left(\int_{\partial{M}\cap M^+}\kappa_gds-\int_{\partial{M}\cap M^-}\kappa_gds\right)+2\pi\left(\#S^+-\#S^-\right)}$

\ \ \ $\displaystyle{-\pi\left(\#(\Sigma\cap\partial{M})^{\mathrm{null}}+2\#(\Sigma\cap\partial{M})^-\right)}$.
\item[$(3)$]$\displaystyle{\int_{\partial{M}\cap M^+}\kappa_gds-\int_{\partial{M}\cap M^-}\kappa_gds-2\pi\left(\chi(M^+)-\chi(M^-)\right)}$

\ \ \ $\displaystyle{-2\pi\left(\#S^+-\#S^-\right)-\pi\left(\#(\Sigma\cap\partial{M})^+-\#(\Sigma\cap\partial{M})^-\right)}$

$\displaystyle{=\int_{\partial{M}\cap M^+_\star}\kappa_g^\star ds_\star-\int_{\partial{M}\cap M^-_\star}\kappa_g^\star ds_\star-2\pi\left(\chi(M^+_\star)-\chi(M^-_\star)\right)}$

\ \ \ $\displaystyle{-2\pi\left(\#S^+_\star-\#S^-_\star\right)-\pi\left(\#(\Sigma_\star\cap\partial{M})^+-\#(\Sigma_\star\cap\partial{M})^-\right)}$.
\item[$(4)$]$\displaystyle{2\int_{\Sigma}\kappa_sds+\int_{\partial{M}}\kappa_gds-2\int_{M^-}Kd\widehat{A}-\sum_{p\in(\Sigma\cap\partial{M})^{\mathrm{null}}}\left(2\alpha^+(p)-\pi\right)}$

$\displaystyle{=2\int_{\Sigma_\star}\kappa_s^\star ds_\star+\int_{\partial{M}}\kappa_g^\star ds_\star-2\int_{M^-_\star}K_\star d\widehat{A}_\star-\sum_{p\in(\Sigma_\star\cap\partial{M})^{\mathrm{null}}}\left(2\alpha^+_\star(p)-\pi\right)}$.
\end{itemize}}
\end{theorem}
The notations in $(1)$-$(4)$ are explained in \S$\ref{sec2}$, $\ref{sec3}$. Theorem $\ref{thm1.1}$ is proved for general frontal bundles in Theorem $\ref{thm4.1}$. Theorem $\ref{thm1.1}$ allows us to generalize the formulas derived in \cite{4} to the case of surfaces with boundary (see Corollary $\ref{cor4.2}$). Furthermore, Theorem $\ref{thm1.1}$ can be applied to surfaces with non-zero bounded extrinsic curvature on the set of regular points of $f$ (see Theorem $\ref{thm5.2}$)．

The paper is organized as follows: In \S $\ref{sec2}$, we briefly review the notion of a coherent tangent bundle over a surface possibly with boundary, and two Gauss-Bonnet type formulas for it. In \S $\ref{sec3}$, we give the definition of a frontal bundle over a surface possibly with boundary, as well as definitions of some notions related to the frontal bundle, such as extrinsic curvature. We also derive Gauss-Bonnet type formulas for frontal bundles based on \S $\ref{sec2}$. In \S $\ref{sec4}$, we derive new formulas by appropriately combining the Gauss-Bonnet type formulas obtained in \S $\ref{sec3}$ (see Theorem $\ref{thm4.1}$). In \S $\ref{sec5}$, we consider the case where the extrinsic curvature of the set of regular points of a frontal is non-zero bounded. Then, we show that the formula for the number of singular points and the formula for the Euler characteristic are derived by Theorem $\ref{thm4.1}$ (see Theorem $\ref{thm5.2}$).

\section{Coherent tangent bundles and Gauss-Bonnet type formulas}\label{sec2}

In this section, we define the structure of a coherent tangent bundle over a surface (possibly with boundary) and let us describe Gauss-Bonnet type formulas for it, based on \cite{1,3,4,6}.

Let $M$ be an oriented surface (possibly with boundary).

\begin{definition}
\label{def2.1}

{\rm A \textit{coherent tangent bundle} over $M$ is a $5$-tuple $(M,\mathcal{E},\langle\cdot,\cdot\rangle,D,\varphi)$ satisfying the following properties:
\begin{itemize}
\item[$(1)$]$\mathcal{E}$ is an orientable vector bundle of rank $2$ over $M$,
\item[$(2)$]$\langle\cdot,\cdot\rangle$ is a metric on $\mathcal{E}$ and $D$ is a metric connection on $(\mathcal{E},\langle\cdot,\cdot\rangle)$,
\item[$(3)$]$\varphi:TM\to\mathcal{E}$ is a bundle homomorphism, such that for any smooth vector fields $X,Y$ on $M$,
\begin{equation*}
D_X\varphi(Y)-D_Y\varphi(X)=\varphi([X,Y])\label{2.1}
\end{equation*}
holds.
\end{itemize}}
\end{definition}

We give a typical example which has the structure of a coherent tangent bundle.

\begin{example}[\mbox{\cite[Example $2.4$]{4}}]
\label{ex2.2}

{\rm Let $(N,g)$ be a $3$-dimensional Riemannian manifold. A $C^\infty$-map $f:M\to N$ is a \textit{frontal} if for each $p\in M$, there exist a neighborhood $U$ of $p$ and a unit vector field $\nu$ along $f$ defined on $U$ such that for any smooth vector field $X$ on $U$, $g\left(df(X),\nu\right)=0$ holds. $\nu$ is called the \textit{unit normal vector field} of $f$ on $U$. Furthermore, if $(f,\nu):U\to T_1N$ is an immersion on a neighborhood $U$ of each $p\in M$, then $f$ is called a \textit{front}, or \textit{wave front}, where $T_1N$ is the unit tangent bundle of $N$. 

Let $f^*TN$ be the pull-back of the tangent bundle $TN$ by a frontal $f:M\to N$. Then, the structure of a coherent tangent bundle can be put into a subbundle $\mathcal{E}_f$ of $f^*TN$ consisting of vectors orthogonal to $\nu$ as follows: A map $\varphi_f:TM\to\mathcal{E}_f,\ \varphi_f(X):=df(X)$ gives a bundle homomorphism, where $df$ is a differential map of $f$. We define a metric on $\mathcal{E}_f$ as the metric induced from the Riemannian metric $g$ of $N$. Then, the tangent component $D$ of the Levi-Civita connection on $N$ gives the metric connection on $\mathcal{E}_f$, and for any vector fields $X,Y$ on $M$,
\begin{equation*}
D_X\varphi_f(Y)-D_Y\varphi_f(X)-\varphi_f([X,Y])=0\label{2.2}
\end{equation*}
holds. Thus, we obtain a coherent tangent bundle $(M,\mathcal{E}_f,\langle\cdot,\cdot\rangle,D,\varphi_f)$ over $M$.

When $N$ is orientable, the assumption that $\mathcal{E}_f$ is orientable leads to $f$ being co-orientable, where $f$ is \textit{co-orientable} if there exists a unit normal vector field $\nu$ globally defined on $M$.}
\end{example}

We fix a coherent tangent bundle $(M,\mathcal{E},\langle\cdot,\cdot\rangle,D,\varphi)$ over $M$.

\begin{definition}
\label{def2.3}

{\rm We define the \textit{first fundamental form} $ds^2_\varphi$ as the pull-back by $\varphi$ of the metric $\langle\cdot,\cdot\rangle$:
\begin{equation*}
ds^2_\varphi:=\varphi^*\langle\cdot,\cdot\rangle.\label{2.3}
\end{equation*}

A point $p\in M$ is a \textit{singular point} of $\varphi$ if $\varphi_p:=\varphi|_{T_pM}:T_pM\to\mathcal{E}_p$ is not a bijection, where $\mathcal{E}_p$ is the fiber of $\mathcal{E}$ at $p$. The set of singular points of $\varphi$ is denoted by $\Sigma_\varphi$. On the other hand, $p$ is called a \textit{regular point} of $\varphi$ if $p$ is not a singular point of $\varphi$.}
\end{definition}

We take a local coordinate system $(U;u,v)$ compatible with the orientation of $M$. Also, we take a positive orthonormal frame $\left\{\bm{e}_1,\bm{e}_2\right\}$ on $\mathcal{E}|_U$ and let $\left\{\omega_1,\omega_2\right\}$ its dual basis.

\begin{definition}
\label{def2.4}

{\rm We define the \textit{area form} $dA_{\mathcal{E}}$ of $\mathcal{E}|_U$ by
\begin{equation}
dA_{\mathcal{E}}:=\omega_1\wedge\omega_2.\label{2.4}
\end{equation}
Then, $dA_{\mathcal{E}}$ is globally defined on $\mathcal{E}$, since $dA_{\mathcal{E}}$ is independent of the choice of positive orthonormal frames of $\mathcal{E}|_U$.

We define the \textit{signed area form} $d\widehat{A}_\varphi$ of $M$ as a pull-back of $dA_{\mathcal{E}}$ by $\varphi$:
\begin{equation}
d\widehat{A}_\varphi:=\varphi^*(dA_{\mathcal{E}}).\label{2.5}
\end{equation}
We define the \textit{signed area density function} $\lambda_\varphi$ on $U$ by
\begin{eqnarray}
\lambda_\varphi&:=&d\widehat{A}_\varphi\left(\frac{\partial}{\partial u},\frac{\partial}{\partial v}\right)=dA_{\mathcal{E}}\left(\varphi_u,\varphi_v\right)\nonumber\\
&&\ \ \ \left(\varphi_u:=\varphi\left(\frac{\partial}{\partial u}\right),\ \varphi_v:=\varphi\left(\frac{\partial}{\partial v}\right)\right).\label{2.6}
\end{eqnarray}
By $(\ref{2.6})$, the set of singular points $\Sigma_\varphi\cap U$ can be expressed as
\begin{equation*}
\Sigma_\varphi\cap U=\left\{p\in U\mid\lambda_\varphi(p)=0\right\}.\label{2.7}
\end{equation*}
Also, by $(\ref{2.6})$, we see that $d\widehat{A}_\varphi$ defines a smooth $2$-form on $M$, since
\begin{equation*}
d\widehat{A}_\varphi=\lambda_\varphi du\wedge dv.\label{2.8}
\end{equation*}

We define the (\textit{unsigned}) \textit{area form} $dA_\varphi$ on $U$ by
\begin{equation}
dA_\varphi:=|\lambda_\varphi|du\wedge dv.\label{2.9}
\end{equation}
Then, $dA_\varphi$ defines a continuous $2$-form on $M$, since $dA_\varphi$ is independent of the choice of positive local coordinates of $U$.

Using these $2$-forms $d\widehat{A}_\varphi$ and $dA_\varphi$, we define two subsets $M^+_\varphi$ and $M^-_\varphi$ of $M$ as follows, respectively:
\begin{eqnarray}
M^+_\varphi&:=&\left\{p\in M\backslash\Sigma_\varphi\mid (dA_\varphi)_p=(d\widehat{A}_\varphi)_p\right\},\nonumber\\
M^-_\varphi&:=&\left\{p\in M\backslash\Sigma_\varphi\mid (dA_\varphi)_p=-(d\widehat{A}_\varphi)_p\right\}.\label{2.10}
\end{eqnarray}
Then, we obtain $\Sigma_\varphi=\partial M^+_\varphi=\partial M^-_\varphi$. Using the signed area density function $\lambda_\varphi$ on $U$, we can express $M^+_\varphi$ and $M^-_\varphi$ on $U$ as follows, respectively:
\begin{eqnarray*}
M^+_\varphi\cap U=\left\{p\in U\mid \lambda_\varphi(p)>0\right\},\ M^-_\varphi\cap U=\left\{p\in U\mid\lambda_\varphi(p)<0\right\}.\label{2.11}
\end{eqnarray*}}
\end{definition}

\begin{definition}
\label{def2.5}
{\rm By the definition of the metric connection $D$ of the coherent tangent bundle $(\mathcal{E},\langle\cdot,\cdot\rangle)$, for any smooth vector field $X$ on $U$,
\begin{equation*}
\langle D_X\bm{e}_i,\bm{e}_i\rangle=0\ \ \ (i=1,2)\label{2.12}
\end{equation*}
holds. So, we define the \textit{connection form} $\mu$ on $U$ as the $1$-form on $U$ given by
\begin{equation*}
D_X\bm{e}_1=-\mu(X)\bm{e}_2,\ D_X\bm{e}_2=\mu(X)\bm{e}_1.\label{2.13}
\end{equation*}}
\end{definition}

The exterior derivative $d\mu$ of the connection form $\mu$ defines a $2$-form on $M$, since $d\mu$ is independent of the choice of the positive orthonormal frames of $\mathcal{E}|_U$. If we denote by $K_\varphi$ the Gaussian curvature with respect to $ds^2_\varphi$ defined on the set of regular points $M\backslash\Sigma_\varphi$, then we have
\begin{equation*}
d\mu=K_\varphi d\widehat{A}_\varphi.\label{2.14}
\end{equation*}
Thus, the $2$-form $K_\varphi d\widehat{A}_\varphi$ on $M\backslash\Sigma_\varphi$ can be smoothly extended to the $2$-form $d\mu$ on $M$. Also, by $(\ref{2.10})$, we have
\begin{equation*}
K_\varphi d\widehat{A}_\varphi=\left\{\begin{array}{cl}K_\varphi dA_\varphi&(\mbox{on $M^+_\varphi$}),\\- K_\varphi dA_\varphi&(\mbox{on $M ^-_\varphi$}).\end{array}\right.\label{2.15}
\end{equation*}

\begin{definition}
\label{def2.6}

{\rm A singular point $p\in U$ is \textit{non-degenerate} if the derivative $d\lambda_\varphi$ of the signed area density function $\lambda_\varphi$ does not vanish at $p$. This definition is independent of the choice of the coordinates of $U$.}
\end{definition}

\begin{definition}
\label{def2.7}

{\rm If $p\in U$ is a non-degenerate singular point, taking a neighborhood $U$ of $p$ sufficiently small, we can parameterize $\Sigma_\varphi\cap U$ by a regular curve $\gamma(t)\ (\gamma(0)=p)$. Such a curve is called a \textit{singular curve}, and the direction of the tangent vector $\gamma^\prime(t):=\frac{d\gamma}{dt}(t)$ of $\gamma(t)$ is called the \textit{singular direction}.

On the other hand, if $p$ is a non-degenerate singular point, the kernel $\ker\varphi_p$ of $\varphi_p$ is a $1$-dimensional linear space, which is called the \textit{null direction}. Also, the null direction at each point $\gamma(t)$ is $1$-dimensional because $\gamma$ consists of non-degenerate singular points. Therefore, we can take a smooth vector field $\eta$ along $\gamma$ such that $\eta(t)$ points in the null direction at $\gamma(t)$, which is called a \textit{null vector field}.}
\end{definition}

\begin{definition}
\label{def2.8}

{\rm Let $p\in M$ be a non-degenerate singular point, $\gamma(t)$ be a singular curve passing through $p=\gamma(0)$, and $\eta(t)$ be a null vector field along $\gamma(t)$. $p$ is of the \textit{first} (resp. \textit{second}) \textit{kind} if the singular direction and the null direction at $p$ are different (resp. the same), i.e.,
\begin{equation*}
\det\left(\gamma^\prime(0),\eta(0)\right)\neq0\ \ \ \left(\mbox{resp}.\ \det\left(\gamma^\prime(0),\eta(0)\right)=0\right),\label{2.16}
\end{equation*}
where ``$\det$'' denotes the determinant function of the $2\times2$ matrix obtained by identifying $T_pM$ with $\mathbb{R}^2$. Furthermore, a singular point $p$ of the second kind is \textit{admissible} if there exists a neighborhood $V$ of $p$ such that the set of singular points on $V\backslash\{p\}$ consists only of singular points of the first kind, i.e., there exists a non-negative integer $k$ such that
\begin{equation*}
\delta(0)=0,\ \delta^\prime(0)=0,\ \cdots,\ \delta^{(k)}(0)=0,\ \delta^{(k+1)}(0)\neq0,\label{2.17}
\end{equation*}
where $\delta(t):=\det\left(\gamma^\prime(t),\eta(t)\right)$. We note that an admissible singular point of the second kind is a non-degenerate peak (see the definition of ``peak'' in \cite{1,3,5}).}
\end{definition}

We fix a singular point $p\in M$ of the first kind. Then, there exists a neighborhood $V$ of $p$ and a singular curve $\gamma(t)$ passing through $p=\gamma(0)$ on $V$. If necessary, the neighborhood $V$ of $p$ can be taken small enough so that $\gamma(t)$ consists only of singular points of the first kind. Therefore, we can assume that $\gamma(t)$ consists only of singular points of the first kind. If $\eta(t)$ is a null vector field along $\gamma(t)$, then $\gamma^\prime(t)$ and $\eta(t)$ are linearly independent at $\gamma(t)$. Therefore, $\varphi(\gamma^\prime(t))$ is non-zero. Furthermore, $d\lambda_\varphi(\gamma^\prime(t))$ is zero, since $\gamma(t)$ is consisted of singular points of $\varphi$. Therefore, $d\lambda_\varphi(\eta(t))$ is non-zero.

\begin{definition}
\label{def2.9}

{\rm We define the \textit{singular curvature} $\kappa_s^\varphi(t)$ of the singular curve $\gamma(t)$ as
\begin{eqnarray*}
\kappa_s^\varphi(t):=\mathrm{sgn}\left(d\lambda_\varphi(\eta(t))\right)\frac{\left\langle D_t\varphi(\gamma^\prime(t)),\bm{n}(t)\right\rangle}{|\varphi(\gamma^\prime(t))|^2},\label{2.18}
\end{eqnarray*}
where the null vector field $\eta(t)$ is taken such that $\left\{\gamma^\prime(t),\eta(t)\right\}$ is compatible with the orientation of $M$, $\bm{n}(t)$ is a section of $\mathcal{E}$ along $\gamma(t)$ such that $\left\{\frac{\varphi(\gamma^\prime(t))}{|\varphi(\gamma^\prime(t))|},\bm{n}(t)\right\}$ is a positive orthonormal frame on $\mathcal{E}_{\gamma(t)}$, $D_t\gamma^\prime(t)$ is a covariant derivative of $\gamma^\prime(t)$ with respect to $\gamma^\prime(t)$, and $|\cdot|:=\sqrt{\langle\cdot,\cdot\rangle}$.}
\end{definition}

We give two properties of the singular curvature necessary to show Gauss-Bonnet type formulas.

\begin{proposition}[\mbox{\cite[Proposition $1.7$]{3}}]
\label{prop2.10}

{\it The value of the singular curvature $\kappa_s^\varphi$ is independent of the parameters of the singular curve $\gamma$, the orientation of $\gamma$, the orientation of $M$ and the orientation of $\mathcal{E}$.}
\end{proposition}

\begin{proposition}[\mbox{\cite[Proposition $2.11$]{3}}]
\label{prop2.11}

{\it If the singular point $p$ of the second kind is admissible and $\gamma(t)$ is a singular curve passing through $p$, then the singular curvature measure $\kappa_s^\varphi ds_\varphi$ defines a bounded $1$-form along $\gamma$, where $ds_\varphi:=\left|\varphi(\gamma^\prime(t))\right|dt$.}
\end{proposition}

Let $M$ be a compact oriented surface with boundary and $(M,\mathcal{E},\langle\cdot,\cdot\rangle,D,\varphi)$ a coherent tangent bundle over $M$. We suppose that $\varphi$ admits only singular points of the first kind and admissible singular points of the second kind, and that the set of singular points $\Sigma_\varphi$ is transversal to the boundary $\partial M$. We triangulate $M$ so that the singular points of the second kind in the interior $M\backslash\partial M$ and the singular points on $\partial M$ are vertices.

\begin{proposition}[\mbox{\cite[Theorem $\mathrm{A}$]{3}}]
\label{prop2.12}

{\it Let $p\in M\backslash\partial M$ be a singular point of the second kind. Then, the sum $\alpha^+_\varphi(p)$ (resp. $\alpha^-_\varphi(p)$) of interior angles on the $M^+_\varphi$ (resp. $M^-_\varphi$) side at $p$ satisfies
\begin{equation*}
\alpha^+_\varphi(p)+\alpha^-_\varphi(p)=2\pi,\ \alpha^+_\varphi(p)-\alpha^-_\varphi(p)\in\left\{-2\pi,0,2\pi\right\}.\label{2.19}
\end{equation*}}
\end{proposition}

\begin{definition}
\label{def2.13}

{\rm A singular point $p\in M\backslash\partial M$ of the second kind is called \textit{positive} (resp. \textit{null}, \textit{negative}) if
\begin{eqnarray*}
&&\alpha^+_\varphi(p)-\alpha^-_\varphi(p)=2\pi\nonumber\\
&&\ \ \ \left(\mbox{resp}.\ \alpha^+_\varphi(p)-\alpha^-_\varphi(p)=0,\ \alpha^+_\varphi(p)-\alpha^-_\varphi(p)=-2\pi\right).\label{2.20}
\end{eqnarray*}}
\end{definition}

\begin{proposition}[\mbox{\cite[Theorem $2.13$]{1}}]
\label{prop2.14}

{\it Let $p\in\partial M$ be a singular point. If the null direction at $p$ is different from the direction of $\partial M$, then $\alpha^+_\varphi(p)$ and $\alpha^-_\varphi(p)$ satisfy
\begin{equation*}
\alpha^+_\varphi(p)+\alpha^-_\varphi(p)=\pi,\ \alpha^+_\varphi(p)-\alpha^-_\varphi(p)\in\left\{-\pi,\pi\right\}.\label{2.21}
\end{equation*}
On the other hand, if the null direction at $p$ is the same as the direction of $\partial M$, then $\alpha^+_\varphi(p)$ and $\alpha^-_\varphi(p)$ satisfy
\begin{equation*}
\alpha^+_\varphi(p)-\alpha^-_\varphi(p)=0.\label{2.22}
\end{equation*}}
\end{proposition}

\begin{definition}
\label{def2.15}

{\rm A singular point $p\in\partial M$ is called \textit{positive} (resp. \textit{null}, \textit{negative}) if
\begin{eqnarray*}
&&\alpha^+_\varphi(p)-\alpha^-_\varphi(p)=\pi\nonumber\\
&&\ \ \ \left(\mbox{resp}.\ \alpha^+_\varphi(p)-\alpha^-_\varphi(p)=0,\ \alpha^+_\varphi(p)-\alpha^-_\varphi(p)=-\pi\right).\label{2.23}
\end{eqnarray*}}
\end{definition}

Finally, we describe the Gauss-Bonnet type formulas for coherent tangent bundles over surfaces with boundary.

\begin{proposition}[\mbox{\cite[Theorem $2.20$]{1}}]
\label{prop2.16}

{\it Let $M$ be a compact oriented surface with boundary and $(M,\mathcal{E},\langle\cdot,\cdot\rangle,D,\varphi)$ a coherent tangent bundle over $M$. We suppose that $\varphi$ admits only singular points of the first kind and admissible singular points of the second kind, and that the set of singular points $\Sigma_\varphi$ is transversal to the boundary $\partial M$. Then, we have the following two formulas:
\begin{itemize}
\item[$(1)$]$\displaystyle{2\int_{\Sigma_\varphi}\kappa_s^\varphi ds_\varphi+\int_{\partial{M}}\kappa_g^\varphi ds_\varphi+\int_MK_\varphi dA_\varphi}$

$\displaystyle{=2\pi\chi(M)+\sum_{p\in(\Sigma_\varphi\cap\partial{M})^{\mathrm{null}}}\left(2\alpha^+_\varphi(p)-\pi\right)}$,
\item[$(2)$]$\displaystyle{\int_{\partial{M}\cap M^+_\varphi}\kappa_g^\varphi ds_\varphi-\int_{\partial{M}\cap M^-_\varphi}\kappa_g^\varphi ds_\varphi+\int_{M}K_\varphi d\widehat{A}_\varphi}$

$\displaystyle{=2\pi\left(\chi(M^+_\varphi)-\chi(M^-_\varphi)\right)+2\pi\left(\#S^+_\varphi-\#S^-_\varphi\right)}$

\ \ \ $\displaystyle{+\pi\left(\#(\Sigma_\varphi\cap\partial{M})^+-\#(\Sigma_\varphi\cap\partial{M})^-\right)}$,
\end{itemize}
where $\kappa_g^\varphi$ is the geodesic curvature, $S^+_\varphi$ (resp. $S^-_\varphi$) is the set of positive (resp. negative) singular points of the second kind in $M\backslash\partial M$, and $(\Sigma_\varphi\cap\partial{M})^+$ (resp. $(\Sigma_\varphi\cap\partial{M})^{\mathrm{null}}$, $(\Sigma_\varphi\cap\partial{M})^-$) is the set of positive (resp. null, negative) singular points on $\partial M$.}
\end{proposition}

\begin{remark}
\label{rem2.17}

{\rm We briefly explain the well-definedness of the integral of the geodesic curvature $\kappa_g^\varphi$ on the boundary $\partial M$ in Proposition $\ref{prop2.16}$. The set of singular points $\Sigma_\varphi$ is assumed to be transversal to $\partial M$. Then, the null direction at a singular point $p$ on $\partial M$ may be either different from or the same as the direction of $\partial M$. If these directions are different at $p$, it is easy to check that $\kappa_g^\varphi$ is bounded at $p$, but if the same, $\kappa_g^\varphi$ is not defined at $p$. However, if we consider the geodesic curvature measure $\kappa_g^\varphi ds_\varphi$, we can show that it is a bounded $1$-form at $p$ (see \cite[Proposition $2.19$]{1}). Thus, the integral of $\kappa_g^\varphi$ on $\partial M$ is well-defined.}
\end{remark}

\begin{remark}
\label{rem2.18}

{\rm Proposition $\ref{prop2.16}$ corresponds to the assertion of \cite[Theorem $2.20$]{1} when all peaks are non-degenerate.}
\end{remark}

\section{Frontal bundles and four Gauss-Bonnet type formulas}\label{sec3}

Let $M$ be an oriented surface (possibly with boundary) and $(M,\mathcal{E},\langle\cdot,\cdot\rangle,D,\varphi)$ a coherent tangent bundle over $M$.

\begin{definition}
\label{def3.1}

{\rm If there exists another bundle homomorphism $\psi:TM\to\mathcal{E}$ such that $(M,\mathcal{E},\langle\cdot,\cdot\rangle,D,\psi)$ is a coherent tangent bundle over $M$ and the pair $(\varphi,\psi)$ of bundle homomorphisms satisfies
\begin{equation*}
\langle\varphi(X),\psi(Y)\rangle=\langle\varphi(Y),\psi(X)\rangle\ \ \ (X,Y\in TM),\label{3.1}
\end{equation*}
then a $6$-tuple $(M,\mathcal{E},\langle\cdot,\cdot\rangle,D,\varphi,\psi)$ is called a \textit{frontal bundle} over $M$ The bundle homomorphisms $\varphi$ and $\psi$ are called the \textit{first homomorphism} and the \textit{second homomorphism}, respectively.

Furthermore, a frontal bundle $(M,\mathcal{E},\langle\cdot,\cdot\rangle,D,\varphi,\psi)$ is a \textit{front bundle} if for each $p\in M$,
\begin{equation*}
\left(\ker\varphi_p\right)\cap\left(\ker\psi_p\right)=\{0\}\label{3.2}
\end{equation*}
holds.}
\end{definition}

We give a typical example with the structure of a frontal bundle.

\begin{example}[\mbox{\cite[Example $2.14$]{4}}]
\label{ex3.2}

{\rm Let $(N,g)$ be a Riemannian $3$-manifold and $\nabla$ a Levi-Civita connection on $N$. If $f:M\to N$ is a co-orientable frontal, then there exists a globally defined unit normal vector field $\nu:M\to T_1N$ on $M$. Since the coherent tangent bundle $(M,\mathcal{E}_f,\langle\cdot,\cdot\rangle,D,\varphi_f)$ given in Example $\ref{ex2.2}$ is orthogonal to $\nu$, a bundle homomorphism $\psi_f:TM\to\mathcal{E}_f,\ \psi_f(X):=D_X\nu$ is defined. We can check that these bundle homomorphisms $\varphi_f$ and $\psi_f$ satisfy (see \cite[Proposition $2.4$]{2})
\begin{equation*}
\langle\varphi_f(X),\psi_f(Y)\rangle=\langle\varphi_f(Y),\psi_f(X)\rangle\ \ \ (X,Y\in TM).\label{3.3}
\end{equation*}
Hence, $(M,\mathcal{E}_f,\langle\cdot,\cdot\rangle,D,\varphi_f,\psi_f)$ is a frontal bundle over $M$.

Furthermore, a frontal bundle $(M,\mathcal{E}_f,\langle\cdot,\cdot\rangle,D,\varphi_f,\psi_f)$ being a front bundle is equivalent to a frontal $f:M\to N$ being a front.}
\end{example}

We fix a frontal bundle $(M,\mathcal{E},\langle\cdot,\cdot\rangle,D,\varphi,\psi)$ over $M$.

\begin{definition}
\label{def3.3}

{\rm We define symmetric covariant tensor fields $\mathrm{I}$, $\mathrm{II}$, and $\mathrm{III}$ on $M$ as
\begin{eqnarray*}
\mathrm{I}(X,Y)&:=&ds^2_\varphi(X,Y)=\langle\varphi(X),\varphi(Y)\rangle,\nonumber\\
\mathrm{II}(X,Y)&:=&-\langle\varphi(X),\psi(Y)\rangle=\left(-\langle\varphi(Y),\psi(X)\rangle\right),\nonumber\\
\mathrm{III}(X,Y)&:=&ds_\psi^2(X,Y)=\langle\psi(X),\psi(Y)\rangle,\label{3.4}
\end{eqnarray*}
respectively, for any vector fields $X,Y$ on $M$. We call $\mathrm{I}$, $\mathrm{II}$, and $\mathrm{III}$ the \textit{first fundamental form}, the \textit{second fundamental form}, and the \textit{third fundamental form}, respectively.

When $p\in M$ is a regular point of $\varphi$, we define the \textit{extrinsic curvature} $K^{\mathrm{ext}}$ by
\begin{eqnarray}
K^{\mathrm{ext}}:=\frac{\mathrm{II}(X,X)\mathrm{II}(Y,Y)-\mathrm{II}(X,Y)^2}{\mathrm{I}(X,X)\mathrm{I}(Y,Y)-\mathrm{I}(X,Y)^2}\ \ \ (X,Y\in T_pM).\label{3.5}
\end{eqnarray}}
\end{definition}

We fix a local coordinate system $(U;u,v)$ compatible with the orientation of $M$.

We denote by $\lambda$ (resp. $\lambda_\star$) the area density function with respect to $\varphi$ (resp. $\psi$) (see $(\ref{2.6})$):
\begin{equation*}
\lambda:=\lambda_\varphi=dA_{\mathcal{E}}\left(\varphi_u,\varphi_v\right),\ \lambda_\star:=\lambda_\psi=dA_{\mathcal{E}}\left(\psi_u,\psi_v\right),\label{3.6}
\end{equation*}
where $\varphi_u=\varphi\left(\frac{\partial}{\partial u}\right)$, $\varphi_v=\varphi\left(\frac{\partial}{\partial v}\right)$, $\psi_u=\psi\left(\frac{\partial}{\partial u}\right)$, $\psi_v=\psi\left(\frac{\partial}{\partial v}\right)$, and $dA_{\mathcal{E}}$ is the area form of $\mathcal{E}$ (see $(\ref{2.4})$). Then, the set of singular points $\Sigma:=\Sigma_\varphi$ (resp. $\Sigma_\star:=\Sigma_\psi$) is represented on $U$ as
\begin{equation*}
\Sigma\cap U=\left\{p\in U\mid\lambda(p)=0\right\},\ \Sigma_\star\cap U=\left\{p\in U\mid\lambda_\star(p)=0\right\}.\label{3.7}
\end{equation*}

We denote by $d\widehat{A}$ and $dA$ (resp. $d\widehat{A}_\star$ and $dA_\star$) the signed area and (unsigned) area forms with respect to $\varphi$ (resp. $\psi$), respectively (see $(\ref{2.5})$, $(\ref{2.9})$):
\begin{gather*}
d\widehat{A}:=d\widehat{A}_\varphi=\varphi^*(dA_{\mathcal{E}}),\ dA:=dA_\varphi=|\lambda|du\wedge dv,\nonumber\\
d\widehat{A}_\star:=d\widehat{A}_\psi=\psi^*(dA_{\mathcal{E}}),\ dA_\star:=dA_\psi=|\lambda_\star|du\wedge dv.\label{3.8}
\end{gather*}
The two subsets $M^\pm:=M^\pm_\varphi$ (resp. $M^\pm_\star:=M^\pm_\psi$) of $M$ defined by $d\widehat{A}$ and $dA$ (resp. $d\widehat{A}_\star$ and $dA_\star$)  are expressed as (see $(\ref{2.10})$)
\begin{eqnarray*}
&&M^\pm:=\left\{p\in M\backslash\Sigma\mid d\widehat{A}=\pm dA\right\}\nonumber\\
&&\ \ \ \left(\mbox{resp}.\ M^\pm_\star:=\left\{p\in M\backslash\Sigma_\star\mid d\widehat{A}_\star=\pm dA_\star\right\}\right).\label{3.9}
\end{eqnarray*}

We denote by $K$ (resp. $K_\star $) the Gaussian curvature with respect to $\mathrm{I}$ (resp. $\mathrm{III}$) defined on $M\backslash\Sigma$ (resp. $M\backslash\Sigma_\star$). Also, we denote by $K^{\mathrm{ext}}$ (resp. $K^{\mathrm{ext}}_\star$) the extrinsic curvature with respect to $\varphi$ (resp. $\psi$) (see $(\ref{3.5})$):
\begin{eqnarray*}
K^{\mathrm{ext}}&:=&\frac{\mathrm{II}(X,X)\mathrm{II}(Y,Y)-\mathrm{II}(X,Y)^2}{\mathrm{I}(X,X)\mathrm{I}(Y,Y)-\mathrm{I}(X,Y)^2}\ \ \ (p\in M\backslash\Sigma,\ X,Y\in T_pM),\nonumber\\
K^{\mathrm{ext}}_\star&:=&\frac{\mathrm{II}(X,X)\mathrm{II}(Y,Y)-\mathrm{II}(X,Y)^2}{\mathrm{III}(X,X)\mathrm{III}(Y,Y)-\mathrm{III}(X,Y)^2}\ \ \ (p\in M\backslash\Sigma_\star,\ X,Y\in T_pM).\label{3.10}
\end{eqnarray*}

The following assertions regarding area forms, the Gaussian curvatures, and the extrinsic curvatures hold.

\begin{lemma}[\mbox{\cite[Lemma $3.10$]{4}}]
\label{lem3.4}

{\it The following holds:
\begin{itemize}
\item[$(1)$]$Kd\widehat{A}=K_\star d\widehat{A}_\star$,
\item[$(2)$]$K^{\mathrm{ext}}d\widehat{A}=d\widehat{A}_\star$ and $K^{\mathrm{ext}}_\star d\widehat{A}_\star=d\widehat{A}$,
\item[$(3)$]$|K^{\mathrm{ext}}|dA=dA_\star$ and $|K^{\mathrm{ext}}_\star|dA_\star=dA$,
\item[$(4)$]$K^{\mathrm{ext}}K^{\mathrm{ext}}_\star=1$.
\end{itemize}}
\end{lemma}

Finally, we describe four Gauss-Bonnet type formulas obtained naturally from a frontal bundle.

\begin{theorem}[The Gauss-Bonnet type formulas for frontal bundles]
\label{thm3.5}

{\it Let $M$ be a compact oriented surface with boundary and $(M,\mathcal{E},\langle\cdot,\cdot\rangle,D,\varphi,\psi)$ a frontal bundle over $M$. We suppose that $\varphi$ and $\psi$ admit only singular points of the first kind and admissible singular points of the second kind, and that the set of singular points $\Sigma$ and $\Sigma_\star$ are transversal to the boundary $\partial M$. Then, we have the following four formulas:
\begin{itemize}
\item[$(1)$]$\displaystyle{2\int_{\Sigma}\kappa_sds+\int_{\partial{M}}\kappa_gds+\int_MKdA}$

$\displaystyle{=2\pi\chi(M)+\sum_{p\in(\Sigma\cap\partial{M})^{\mathrm{null}}}\left(2\alpha^+(p)-\pi\right)}$,
\item[$(2)$]$\displaystyle{\int_{\partial{M}\cap M^+}\kappa_gds-\int_{\partial{M}\cap M^-}\kappa_gds+\int_MKd\widehat{A}}$

$\displaystyle{=2\pi\left(\chi(M^+)-\chi(M^-)\right)+2\pi\left(\#S^+-\#S^-\right)}$

\ \ \ $\displaystyle{+\pi\left(\#(\Sigma\cap\partial{M})^+-\#(\Sigma\cap\partial{M})^-\right)}$,
\item[$(3)$]$\displaystyle{2\int_{\Sigma_\star}\kappa_s^\star ds_\star+\int_{\partial{M}}\kappa_g^\star ds_\star+\int_MK_\star dA_\star}$

$\displaystyle{=2\pi\chi(M)+\sum_{p\in(\Sigma_\star\cap\partial{M})^{\mathrm{null}}}\left(2\alpha^+_\star(p)-\pi\right)}$,
\item[$(4)$]$\displaystyle{\int_{\partial{M}\cap M^+_\star}\kappa_g^\star ds_\star-\int_{\partial{M}\cap M^-_\star}\kappa_g^\star ds_\star+\int_MK_\star d\widehat{A}_\star}$

$\displaystyle{=2\pi\left(\chi(M^+_\star)-\chi(M^-_\star)\right)+2\pi\left(\#S^+_\star-\#S^-_\star\right)}$

\ \ \ $\displaystyle{+\pi\left(\#(\Sigma_\star\cap\partial{M})^+-\#(\Sigma_\star\cap\partial{M})^-\right)}$,
\end{itemize}
where $\kappa_sds$ (resp. $\kappa_s^\star ds_\star$) is the singular curvature measure along $\Sigma$ (resp. $\Sigma_\star$), $\kappa_gds$ (resp. $\kappa_g^\star ds_\star$) is the geodesic curvature measure with respect to $\varphi$ (resp. $\psi$), $S^+$ (resp. $S^+_\star$) is the set of positive singular points of the second kind in $\Sigma\backslash\partial{M}$ (resp. $\Sigma_\star\backslash\partial{M}$), $S^-$ (resp. $S^-_\star$) is the set of negative singular points of the second kind in $\Sigma\backslash\partial{M}$ (resp. $\Sigma_\star\backslash\partial{M}$), $(\Sigma\cap\partial{M})^+$ (resp. $(\Sigma_\star\cap\partial{M})^+$) is the set of positive singular points on $\Sigma\cap\partial{M}$ (resp. $\Sigma_\star\cap\partial{M}$), $(\Sigma\cap\partial{M})^-$ (resp. $(\Sigma_\star\cap\partial{M})^-$) is the set of negative singular points on $\Sigma\cap\partial{M}$ (resp. $\Sigma_\star\cap\partial{M}$), $(\Sigma\cap\partial{M})^{\mathrm{null}}$ (resp. $(\Sigma_\star\cap\partial{M})^{\mathrm{null}}$) is the set of null singular points on $\Sigma\cap\partial{M}$ (resp. $\Sigma_\star\cap\partial{M}$), and $\alpha^+(p)$ (resp. $\alpha^+_\star(p)$) is the sum of the interior angles on the $M^+$ (resp. $M^+_\star$) side at $p\in(\Sigma\cap\partial{M})^{\mathrm{null}}$ (resp. $p\in(\Sigma_\star\cap\partial{M})^{\mathrm{null}}$).}
\end{theorem}

\begin{proof}

A frontal bundle $(M,\mathcal{E},\langle\cdot,\cdot\rangle,D,\varphi,\psi)$ consists of two coherent tangent bundles $(M,\mathcal{E},\langle\cdot,\cdot\rangle,D,\varphi)$ and $(M,\mathcal{E},\langle\cdot,\cdot\rangle,D,\psi)$ over $M$. Therefore, by applying Proposition $\ref{prop2.16}$ to each coherent tangent bundle, we obtain the formulas $(1)$-$(4)$.\qed
\end{proof}

\section{New formulas derived from four Gauss-Bonnet type formulas}\label{sec4}

By appropriately combining the Gauss-Bonnet type formulas in Theorem $\ref{thm3.5}$, we obtain the following assertion.

\begin{theorem}
\label{thm4.1}

{\it Let $M$ be a compact oriented surface with boundary and $(M,\mathcal{E},\langle\cdot,\cdot\rangle,D,\varphi,\psi)$ a frontal bundle over $M$. We suppose that $\varphi$ and $\psi$ allow only singular points of the first kind and admissible singular points of the second kind, and that the sets of singular points $\Sigma:=\Sigma_\varphi$ and $\Sigma_\star:=\Sigma_\psi$ are transversal to the boundary $\partial{M}$. Then, we have the following four formulas:
\begin{itemize}
\item[$(1)$]$\displaystyle{4\pi\chi(M^-_\star)=2\int_{\Sigma}\kappa_sds+\int_{\partial{M}}\kappa_gds}$

\ \ \ $\displaystyle{-2\int_{M^-}Kd\widehat{A}-\sum_{p\in(\Sigma\cap\partial{M})^{\mathrm{null}}}\left(2\alpha^+(p)-\pi\right)}$

\ \ \ $\displaystyle{-\left(\int_{\partial{M}\cap M^+_\star}\kappa_g^\star ds_\star-\int_{\partial{M}\cap M^-_\star}\kappa_g^\star ds_\star\right)+2\pi\left(\#S^+_\star-\#S^-_\star\right)}$

\ \ \ $\displaystyle{-\pi\left(\#(\Sigma_\star\cap\partial{M})^{\mathrm{null}}+2\#(\Sigma_\star\cap\partial{M})^-\right)}$,
\item[$(2)$]$\displaystyle{4\pi\chi(M^-)=2\int_{\Sigma_\star}\kappa_s^\star ds_\star+\int_{\partial{M}}\kappa_g^\star ds_\star}$

\ \ \ $\displaystyle{-2\int_{M^-_\star}K_\star d\widehat{A}_\star-\sum_{p\in(\Sigma_\star\cap\partial{M})^{\mathrm{null}}}\left(2\alpha^+_\star(p)-\pi\right)}$

\ \ \ $\displaystyle{-\left(\int_{\partial{M}\cap M^+}\kappa_gds-\int_{\partial{M}\cap M^-}\kappa_gds\right)+2\pi\left(\#S^+-\#S^-\right)}$

\ \ \ $\displaystyle{-\pi\left(\#(\Sigma\cap\partial{M})^{\mathrm{null}}+2\#(\Sigma\cap\partial{M})^-\right)}$,
\item[$(3)$]$\displaystyle{\int_{\partial{M}\cap M^+}\kappa_gds-\int_{\partial{M}\cap M^-}\kappa_gds-2\pi\left(\chi(M^+)-\chi(M^-)\right)}$

\ \ \ $\displaystyle{-2\pi\left(\#S^+-\#S^-\right)-\pi\left(\#(\Sigma\cap\partial{M})^+-\#(\Sigma\cap\partial{M})^-\right)}$

$\displaystyle{=\int_{\partial{M}\cap M^+_\star}\kappa_g^\star ds_\star-\int_{\partial{M}\cap M^-_\star}\kappa_g^\star ds_\star-2\pi\left(\chi(M^+_\star)-\chi(M^-_\star)\right)}$

\ \ \ $\displaystyle{-2\pi\left(\#S^+_\star-\#S^-_\star\right)-\pi\left(\#(\Sigma_\star\cap\partial{M})^+-\#(\Sigma_\star\cap\partial{M})^-\right)}$,
\item[$(4)$]$\displaystyle{2\int_{\Sigma}\kappa_sds+\int_{\partial{M}}\kappa_gds-2\int_{M^-}Kd\widehat{A}-\sum_{p\in(\Sigma\cap\partial{M})^{\mathrm{null}}}\left(2\alpha^+(p)-\pi\right)}$

$\displaystyle{=2\int_{\Sigma_\star}\kappa_s^\star ds_\star+\int_{\partial{M}}\kappa_g^\star ds_\star-2\int_{M^-_\star}K_\star d\widehat{A}_\star-\sum_{p\in(\Sigma_\star\cap\partial{M})^{\mathrm{null}}}\left(2\alpha^+_\star(p)-\pi\right)}$,
\end{itemize}
where the notations in the formulas $(1)$-$(4)$ are the same as those in Theorem $\ref{thm3.5}$.}
\end{theorem}

\begin{proof}

First, we show the formula $(1)$. Since the set of singular points $\Sigma_\star$ consists of non-degenerate singular points and is transversal to the boundary $\partial{M}$, we have
\begin{eqnarray}
\chi(M)=\chi(M_\star^+)+\chi(M_\star^-)+\chi(\Sigma_\star)=\chi(M_\star^+)+\chi(M_\star^-)+\frac{1}{2}\#(\Sigma_\star\cap\partial{M}).\label{4.1}
\end{eqnarray}
Hence, $\displaystyle{\int_MKdA}$ is computed as follows:
\begin{eqnarray}
\int_MKdA&=&\int_MKd\widehat{A}-2\int_{M^-}Kd\widehat{A}\nonumber\\
&=&\int_MK_\star d\widehat{A}_\star-2\int_{M^-}Kd\widehat{A}\ \ \ (\mbox{by $(1)$ of Lemma $\ref{lem3.4}$})\nonumber\\
&=&-2\int_{M^-}Kd\widehat{A}-\left(\int_{\partial{M}\cap M^+_\star}\kappa_g^\star ds_\star-\int_{\partial{M}\cap M^-_\star}\kappa_g^\star ds_\star\right)\nonumber\\
&&\ \ \ +2\pi\left(\chi(M^+_\star)-\chi(M^-_\star)\right)+2\pi\left(\#S^+_\star-\#S^-_\star\right)\nonumber\\
&&\ \ \ +\pi\left(\#(\Sigma_\star\cap\partial{M})^+-\#(\Sigma_\star\cap\partial{M})^-\right)\ \ \ \left(\mbox{by $(4)$ of Theorem $\ref{thm3.5}$}\right)\nonumber\\
&=&-2\int_{M^-}Kd\widehat{A}-\left(\int_{\partial{M}\cap M^+_\star}\kappa_g^\star ds_\star-\int_{\partial{M}\cap M^-_\star}\kappa_g^\star ds_\star\right)\nonumber\\
&&\ \ \ +2\pi\left(\chi(M)-2\chi(M^-_\star)-\frac{1}{2}\#(\Sigma_\star\cap\partial{M})\right)+2\pi\left(\#S^+_\star-\#S^-_\star\right)\nonumber\\
&&\ \ \ +\pi\left(\#(\Sigma_\star\cap\partial{M})^+-\#(\Sigma_\star\cap\partial{M})^-\right)\ \ \ (\mbox{by $(\ref{4.1})$})\nonumber\\
&=&-2\int_{M^-}Kd\widehat{A}-\left(\int_{\partial{M}\cap M^+_\star}\kappa_g^\star ds_\star-\int_{\partial{M}\cap M^-_\star}\kappa_g^\star ds_\star\right)\nonumber\\
&&\ \ \ +2\pi\chi(M)-4\pi\chi(M^-_\star)+2\pi\left(\#S^+_\star-\#S^-_\star\right)\nonumber\\
&&\ \ \ -\pi\left(\#(\Sigma_\star\cap\partial{M})^{\mathrm{null}}+2\#(\Sigma_\star\cap\partial{M})^-\right).\label{4.2}
\end{eqnarray}
Thus, by $(1)$ of Theorem $\ref{thm3.5}$ and $(\ref{4.2})$, we have
\begin{eqnarray*}
2\pi\chi(M)&=&2\int_{\Sigma}\kappa_sds+\int_{\partial{M}}\kappa_gds-\sum_{p\in(\Sigma\cap\partial{M})^{\mathrm{null}}}\left(2\alpha^+(p)-\pi\right)\nonumber\\
&&\ \ \ -2\int_{M^-}Kd\widehat{A}-\left(\int_{\partial{M}\cap M^+_\star}\kappa_g^\star ds_\star-\int_{\partial{M}\cap M^-_\star}\kappa_g^\star ds_\star\right)\nonumber\\
&&\ \ \ +2\pi\chi(M)-4\pi\chi(M^-_\star)+2\pi\left(\#S^+_\star-\#S^-_\star\right)\nonumber\\
&&\ \ \ -\pi\left(\#(\Sigma_\star\cap\partial{M})^{\mathrm{null}}+2\#(\Sigma_\star\cap\partial{M})^-\right).\label{4.3}
\end{eqnarray*}
Thus, by eliminating $2\pi\chi(M)$ on both sides, we obtain the formula $(1)$.

The formula $(2)$ is obtained by exchanging the roles of $\varphi$ and $\psi$ in $(1)$.

We show the formula $(3)$. By $(2)$ and $(4)$ of Theorem $\ref{thm3.5}$, we have
\begin{eqnarray}
\int_MKd\widehat{A}&=&-\left(\int_{\partial{M}\cap M^+}\kappa_gds-\int_{\partial{M}\cap M^-}\kappa_gds\right)\nonumber\\
&&\ \ \ +2\pi\left(\chi(M^+)-\chi(M^-)\right)+2\pi\left(\#S^+-\#S^-\right)\nonumber\\
&&\ \ \ +\pi\left(\#(\Sigma\cap\partial{M})^+-\#(\Sigma\cap\partial{M})^-\right),\label{4.4}\\
\int_MK_\star d\widehat{A}_\star&=&-\left(\int_{\partial{M}\cap M^+_\star}\kappa_g^\star ds_\star-\int_{\partial{M}\cap M^-_\star}\kappa_g^\star ds_\star\right)\nonumber\\
&&\ \ \ +2\pi\left(\chi(M^+_\star)-\chi(M^-_\star)\right)+2\pi\left(\#S^+_\star-\#S^-_\star\right)\nonumber\\
&&\ \ \ +\pi\left(\#(\Sigma_\star\cap\partial{M})^+-\#(\Sigma_\star\cap\partial{M})^-\right).\label{4.5}
\end{eqnarray}
By $(1)$ of Lemma $\ref{lem3.4}$, $(\ref{4.4})$ and $(\ref{4.5})$ are equal, so $(3)$ is obtained.

Finally, we show the formula $(4)$. By $(1)$ and $(3)$ of Theorem $\ref{thm3.5}$, $\int_MKdA=\int_MKd\widehat{A}-2\int_{M^-}Kd\widehat{A}$ and $\int_MK_\star dA_\star=\int_MK_\star d\widehat{A}_\star-2\int_{M^-_\star}K_\star d\widehat{A}_\star$, we have
\begin{eqnarray}
\int_MKd\widehat{A}&=&-2\int_{\Sigma}\kappa_sds-\int_{\partial{M}}\kappa_gds+2\int_{M^-}Kd\widehat{A}\nonumber\\
&&\ \ \ +2\pi\chi(M)+\sum_{p\in(\Sigma\cap\partial{M})^{\mathrm{null}}}\left(2\alpha^+(p)-\pi\right),\label{4.6}\\
\int_MK_\star d\widehat{A}_\star&=&-2\int_{\Sigma_\star}\kappa_s^\star ds_\star-\int_{\partial{M}}\kappa_g^\star ds_\star+2\int_{M^-_\star}K_\star d\widehat{A}_\star\nonumber\\
&&\ \ \ +2\pi\chi(M)+\sum_{p\in(\Sigma_\star\cap\partial{M})^{\mathrm{null}}}\left(2\alpha^+_\star(p)-\pi\right).\label{4.7}
\end{eqnarray}
By $(1)$ of Lemma $\ref{lem3.4}$, $(\ref{4.6})$ and $(\ref{4.7})$ are equal, so $(4)$ is obtained.\qed
\end{proof}

If either the first homomorphism $\varphi$ or the third homomorphism $\psi$ has no singular points, the formulas in Theorem $\ref{thm4.1}$ are reduced to the following formulas.

\begin{corollary}
\label{cor4.2}

{\it Under the assumption of Theorem $\ref{thm4.1}$, if $\varphi$ has no singular points, then we have the following two formulas:
\begin{itemize}
\item[$(1)$]$\displaystyle{2\chi(M^-_\star)=\frac{1}{2\pi}\int_{\partial{M}}\kappa_gds-\frac{1}{2\pi}\left(\int_{\partial{M}\cap M^+_\star}\kappa_g^\star ds_\star-\int_{\partial{M}\cap M^-_\star}\kappa_g^\star ds_\star\right)}$

\ \ \ $\displaystyle{+\#S^+_\star-\#S^-_\star-\left(\frac{\#(\Sigma_\star\cap\partial{M})^{\mathrm{null}}}{2}+\#(\Sigma_\star\cap\partial{M})^-\right)}$,
\item[$(2)$]$\displaystyle{\int_{\partial{M}}\kappa_gds=2\int_{\Sigma_\star}\kappa_s^\star ds_\star+\int_{\partial{M}}\kappa_g^\star ds_\star-2\int_{M^-_\star}K_\star d\widehat{A}_\star-\sum_{p\in(\Sigma_\star\cap\partial{M})^{\mathrm{null}}}\left(2\alpha^+_\star(p)-\pi\right)}$.
\end{itemize}
On the other hand, if $\psi$ has no singular points, then we have the following two formulas:
\begin{itemize}
\item[$(3)$]$\displaystyle{2\chi(M^-)=\frac{1}{2\pi}\int_{\partial{M}}\kappa_g^\star ds_\star-\frac{1}{2\pi}\left(\int_{\partial{M}\cap M^+}\kappa_gds-\int_{\partial{M}\cap M^-}\kappa_gds\right)}$

\ \ \ $\displaystyle{+\#S^+-\#S^--\left(\frac{\#(\Sigma\cap\partial{M})^{\mathrm{null}}}{2}+\#(\Sigma\cap\partial{M})^-\right)}$,
\item[$(4)$]$\displaystyle{\int_{\partial{M}}\kappa_g^\star ds_\star=2\int_{\Sigma}\kappa_sds+\int_{\partial{M}}\kappa_gds-2\int_{M^-}Kd\widehat{A}-\sum_{p\in(\Sigma\cap\partial{M})^{\mathrm{null}}}\left(2\alpha^+(p)-\pi\right)}$.
\end{itemize}}
\end{corollary}

\begin{example}
\label{ex4.3}

{\rm Let $M$ be a compact oriented surface with boundary, $(N,g)$ a $3$-dimensional Riemannian manifold, and $f:M\to N$ a co-orientable frontal. Then, we can construct the frontal bundle $(M,\mathcal{E}_f,\langle\cdot,\cdot\rangle,D,\varphi_f,\psi_f)$ over $M$ as in Example $\ref{ex3.2}$. Therefore, the formulas in Corollary $\ref{cor4.2}$ hold. We note that $(1)$ (resp. $(2)$, $(3)$, $(4)$) of Corollary $\ref{cor4.2}$ is a generalization of the formula in \cite[Theorem $3.12$]{4} (resp. \cite[Theorem $3.18$]{4}, \cite[Theorem $3.16$]{4}, \cite[Theorem $3.23$]{4}) to surfaces with boundary.}
\end{example}

\section{Formulas for frontals with the bounded extrinsic curvature}\label{sec5}

Let $M$ be an oriented surface with boundary, $(N,g)$ a $3$-dimensional Riemannian manifold, and $f:M\to N$ a co-orientable frontal. Then, we obtain the frontal bundle $(M,\mathcal{E}_f,\langle\cdot,\cdot\rangle,D,\varphi:=\varphi_f,\psi:=\psi_f)$ over $M$ as in Example $\ref{ex3.2}$. Then, the following lemma holds.

\begin{lemma}
\label{lem5.1}

{\it Let $p\in M$ be a singular point of the first kind of $\varphi$. We suppose that there exists a neighborhood $U$ of $p$ such that $\log|K^{\mathrm{ext}}|$ is bounded on $U\backslash\Sigma$, where $K^{\mathrm{ext}}$ is the extrinsic curvature with respect to $\varphi$. Then, the following holds on $U$:
\begin{itemize}
\item[$(1)$]$\Sigma=\Sigma_\star$,
\item[$(2)$]If $p$ is also a singular point of the first kind of $\psi$, then we have
\begin{equation*}
M^+=M^{\mathrm{sgn}(K^{\mathrm{ext}}|_U)}_\star,\ \kappa_sds=\mathrm{sgn}(K^{\mathrm{ext}}|_U)\kappa_s^\star ds_\star,\label{5.1}
\end{equation*}
\item[$(3)$]If the second fundamental form $\mathrm{II}$ is equal to a constant multiple of the square root of the product of the first and third fundamental forms $\sqrt{\mathrm{I}\cdot\mathrm{III}}$ on $\partial{M}\backslash\Sigma$, then we have
\begin{equation*}
\kappa_gds=\mathrm{sgn}(K^{\mathrm{ext}}|_U)\kappa_g^\star ds_\star\label{5.2}
\end{equation*}
on $\partial{M}\backslash\Sigma$.
\end{itemize}}
\end{lemma}

$(1)$ and $(2)$ of Lemma $\ref{lem5.1}$ are shown in \cite[Lemma $3.25$]{4}. We show $(3)$.

\begin{proof}[\textit{Proof of $(3)$}.]

We take a positive orthonormal frame field $\left\{\bm{e}_1,\bm{e}_2\right\}$ of $\mathcal{E}|_U$ and define matrices $G$ and $G_\star$ as
\begin{equation*}
\left(\varphi_u,\varphi_v\right)=\left(\bm{e}_1,\bm{e}_2\right)G,\ \left(\psi_u,\psi_v\right)=\left(\bm{e}_1,\bm{e}_2\right)G_\star.\label{5.3}
\end{equation*}
Let $\mu$ be a connection form of the connection $D$ with respect to $\left\{\bm{e}_1,\bm{e}_2\right\}$.

We parameterize $\partial{M}\backslash\Sigma\left(=\partial{M}\backslash\Sigma_\star\right)$ by a regular curve $\gamma(t)$. Then, $\varphi(\gamma^\prime(t))$ is non-zero since $\gamma(t)$ is a regular point. So, we define a $C^\infty$-function $\theta(t)$ as
\begin{equation}
\varphi(\gamma^\prime(t))=|\varphi(\gamma^\prime(t))|\left(\cos\theta(t)(\bm{e}_1)_{\gamma(t)}+\sin\theta(t)(\bm{e}_2)_{\gamma(t)}\right).\label{5.4}
\end{equation}
Hence, the unit co-normal vector field $\bm{n}(t)$ along $\gamma(t)$ is expressed as
\begin{equation*}
\bm{n}(t):=-\sin\theta(t)(\bm{e}_1)_{\gamma(t)}+\cos\theta(t)(\bm{e}_2)_{\gamma(t)}.\label{5.5}
\end{equation*}
Then, the covariant derivative $D_t\bm{n}(t)$ of $\bm{n}(t)$ with respect to $\gamma^\prime(t)$ is expressed as
\begin{equation}
D_t\bm{n}(t)=\left(\mu(\gamma^\prime(t))-\theta^\prime(t)\right)\frac{\varphi(\gamma^\prime(t))}{|\varphi(\gamma^\prime(t))|}.\label{5.6}
\end{equation}
By substituting $(\ref{5.6})$ into the definition of the geodesic curvature $\kappa_g$ with respect to $\varphi$, we obtain
\begin{eqnarray}
\kappa_g(t)&:=&\mathrm{sgn}(\lambda(\gamma(t)))\frac{\langle D_t\varphi(\gamma^\prime(t)),\bm{n}(t)\rangle}{|\varphi(\gamma^\prime(t))|^2}\nonumber\\
&=&-\mathrm{sgn}(\lambda(\gamma(t)))\frac{\langle \varphi(\gamma^\prime(t)),D_t\bm{n}(t)\rangle}{|\varphi(\gamma^\prime(t))|^2}\nonumber\\
&=&-\mathrm{sgn}(\lambda(\gamma(t)))\frac{\mu(\gamma^\prime(t))-\theta^\prime(t)}{|\varphi(\gamma^\prime(t))|}.\label{5.7}
\end{eqnarray}
Therefore, by multiplying both sides of $(\ref{5.7})$ by the arc-length measure $ds:=|\varphi(\gamma^\prime(t))|dt$, we obtain
\begin{equation}
\kappa_g(t)ds=-\mathrm{sgn}(\lambda(\gamma(t)))\left(\mu(\gamma^\prime(t))-\theta^\prime(t)\right)dt.\label{5.8}
\end{equation}

Similarly, defining a $C^\infty$-function $\theta_\star(t)$ by
\begin{equation}
\psi(\gamma^\prime(t))=|\psi(\gamma^\prime(t))|\left(\cos\theta_\star(t)(\bm{e}_1)_{\gamma(t)}+\sin\theta_\star(t)(\bm{e}_2)_{\gamma(t)}\right),\label{5.9}
\end{equation}
we have
\begin{eqnarray}
\kappa_g^\star(t)ds_\star&=&-\mathrm{sgn}(\lambda_\star(\gamma(t)))\left(\mu(\gamma^\prime(t))-\theta^\prime_\star(t)\right)dt\ \ \ \left(ds_\star:=|\psi(\gamma^\prime(t))|dt\right).\label{5.10}
\end{eqnarray}

By $(2)$ of Lemma $\ref{lem3.4}$, we have
\begin{equation}
\mathrm{sgn}(\lambda(\gamma(t)))=\mathrm{sgn}(K^{\mathrm{ext}}|_U)\mathrm{sgn}(\lambda_\star(\gamma(t))).\label{5.11}
\end{equation}

By $(\ref{5.4})$ and $(\ref{5.9})$, we have
\begin{eqnarray}
\mathrm{II}(\gamma^\prime(t),\gamma^\prime(t))=-|\varphi(\gamma^\prime(t))||\psi(\gamma^\prime(t))|\cos\left(\theta(t)-\theta_\star(t)\right).\label{5.12}
\end{eqnarray}
On the other hand, since $\mathrm{II}$ is a constant multiple of $\sqrt{\mathrm{I}\cdot\mathrm{III}}$, there exists a constant $c$ such that
\begin{equation}
\mathrm{II}(\gamma^\prime(t),\gamma^\prime(t))=c|\varphi(\gamma^\prime(t))||\psi(\gamma^\prime(t))|.\label{5.13}
\end{equation}
Hence, by $(\ref{5.12})$ and $(\ref{5.13})$, we see that $\theta(t)-\theta_\star(t)$ is a constant, so
\begin{equation}
\theta^\prime(t)=\theta^\prime_\star(t).\label{5.14}
\end{equation}
Thus, by summarizing $(\ref{5.8})$, $(\ref{5.10})$, $(\ref{5.11})$, and $(\ref{5.14})$, we obtain $(3)$.\qed
\end{proof}

\begin{theorem}
\label{thm5.2}

{\it Let $M$ be a compact oriented surface with boundary, $(N,g)$ a $3$-dimensional Riemannian manifold, $f:M\to N$ a co-orientable frontal and $\nu:M\to T_1N$ a unit normal vector field of $f$. We suppose that $f$ and $\nu$ allow only singular points of the first kind and admissible singular points of the second kind, and that the sets of singular points $\Sigma:=\Sigma_f$ and $\Sigma_\star:=\Sigma_\nu$ are transversal to the boundary $\partial M$. Furthermore, we suppose the following conditions:
\begin{itemize}
\item $\log|K^{\mathrm{ext}}|$ is bounded on $M\backslash\Sigma$ (we note that $\Sigma=\Sigma_\star$ by $(1)$ of Lemma $\ref{lem5.1}$),
\item Singular points of the first kind of $f$ are also singular points of the first kind of $\nu$,
\item The second fundamental form $\mathrm{II}$ is equal to a constant multiple of the square root of the product of the first and third fundamental forms $\sqrt{\mathrm{I}\cdot\mathrm{III}}$ on $\partial M\backslash\Sigma$.
\end{itemize}
Then, if $K^{\mathrm{ext}}>0$, then we have the following two formulas:
\begin{itemize}
\item[$(1)$]$\displaystyle{\#S^+-\#S^-+\frac{\#(\Sigma\cap\partial{M})^+-\#(\Sigma\cap\partial{M})^-}{2}}$

$\displaystyle{=\#S^+_\star-\#S^-_\star+\frac{\#(\Sigma_\star\cap\partial{M})^+-\#(\Sigma_\star\cap\partial{M})^-}{2}}$,
\item[$(2)$]$\displaystyle{\sum_{p\in(\Sigma\cap\partial{M})^{\mathrm{null}}}\left(2\alpha^+(p)-\pi\right)=\sum_{p\in(\Sigma_\star\cap\partial{M})^{\mathrm{null}}}\left(2\alpha^+_\star(p)-\pi\right)}$.
\end{itemize}
On the other hand, if $K^{\mathrm{ext}}<0$, then we have the following two formulas:
\begin{itemize}
\item[$(3)$]$\displaystyle{2\left(\chi(M^+)-\chi(M^-)\right)+\#S^+-\#S^-+\frac{\#(\Sigma\cap\partial{M})^+-\#(\Sigma\cap\partial{M})^-}{2}}$

$\displaystyle{=\#S^+_\star-\#S^-_\star+\frac{\#(\Sigma_\star\cap\partial{M})^+-\#(\Sigma_\star\cap\partial{M})^-}{2}}$,
\item[$(4)$]$\displaystyle{4\pi\chi(M)+\sum_{p\in(\Sigma\cap\partial{M})^{\mathrm{null}}}\left(2\alpha^+(p)-\pi\right)+\sum_{p\in(\Sigma_\star\cap\partial{M})^{\mathrm{null}}}\left(2\alpha^+_\star(p)-\pi\right)=0}$.
\end{itemize}
Here, the notations in $(1)$-$(4)$ are the same as those in Theorem $\ref{thm3.5}$.}
\end{theorem}

\begin{proof}

By Example $\ref{ex3.2}$, we obtain the frontal bundle $(M,\mathcal{E}_f,\langle\cdot,\cdot\rangle,D,\varphi,\psi)$ over $M$. Hence, $(3)$ and $(4)$ of Theorem $\ref{thm4.1}$ hold.

By Lemma $\ref{lem3.4}$ and Lemma $\ref{lem5.1}$, $(3)$ and $(4)$ of Theorem $\ref{thm4.1}$ are reduced, respectively, as
\begin{eqnarray}
&&\int_{\partial{M}\cap M^+}\kappa_gds-\int_{\partial{M}\cap M^-}\kappa_gds\nonumber\\
&&\ \ \ -2\pi\left(\chi(M^+)-\chi(M^-)\right)-2\pi\left(\#S^+-\#S^-\right)\nonumber\\
&&\ \ \ -\pi\left(\#(\Sigma\cap\partial{M})^+-\#(\Sigma\cap\partial{M})^-\right)\nonumber\\
&&=\mathrm{sgn}(K^{\mathrm{ext}})\left(\int_{\partial{M}\cap M^{\mathrm{sgn}(K^{\mathrm{ext}})}}\kappa_gds-\int_{\partial{M}\cap M^{-\mathrm{sgn}(K^{\mathrm{ext}})}}\kappa_gds\right)\nonumber\\
&&\ \ \ -2\pi\left(\chi(M^{\mathrm{sgn}(K^{\mathrm{ext}})})-\chi(M^{-\mathrm{sgn}(K^{\mathrm{ext}})})\right)-2\pi\left(\#S^+_\star-\#S^-_\star\right)\nonumber\\
&&\ \ \ -\pi\left(\#(\Sigma_\star\cap\partial{M})^+-\#(\Sigma_\star\cap\partial{M})^-\right),\label{5.15}\\
&&2\int_{\Sigma}\kappa_sds+\int_{\partial{M}}\kappa_gds-2\int_{M^-}Kd\widehat{A}\nonumber\\
&&\ \ \ -\sum_{p\in(\Sigma\cap\partial{M})^{\mathrm{null}}}\left(2\alpha^+(p)-\pi\right)\nonumber\\
&&=\mathrm{sgn}(K^{\mathrm{ext}})\left(2\int_\Sigma\kappa_sds+\int_{\partial{M}}\kappa_gds\right)-2\int_{M^{-\mathrm{sgn}(K^{\mathrm{ext}})}}Kd\widehat{A}\nonumber\\
&&\ \ \ -\sum_{p\in(\Sigma_\star\cap\partial{M})^{\mathrm{null}}}\left(2\alpha^+_\star(p)-\pi\right).\label{5.16}
\end{eqnarray}

If $K^{\mathrm{ext}}>0$, then $(\ref{5.15})$ and $(\ref{5.16})$ lead to $(1)$ and $(2)$, respectively.

We suppose that $K^{\mathrm{ext}}<0$. Then, $(\ref{5.15})$ leads to $(3)$. On the other hand, $(\ref{5.16})$ is expressed as
\begin{eqnarray}
&&4\int_{\Sigma}\kappa_sds+2\int_{\partial{M}}\kappa_gds+2\int_MKdA-\sum_{p\in(\Sigma\cap\partial{M})^{\mathrm{null}}}\left(2\alpha^+(p)-\pi\right)\nonumber\\
&&=-\sum_{p\in(\Sigma_\star\cap\partial{M})^{\mathrm{null}}}\left(2\alpha^+_\star(p)-\pi\right).\label{5.17}
\end{eqnarray}
By applying $(1)$ of Theorem $\ref{thm3.5}$ to $(\ref{5.17})$, we obtain $(4)$.\qed
\end{proof}

In Theorem $\ref{thm5.2}$, if $M$ has no boundary, we obtain the following assertion.

\begin{corollary}[\mbox{\cite[Theorem $3.28$, $3.30$]{4}}]
\label{cor5.3}

{\it Under the assumption of Theorem $\ref{thm5.2}$, we suppose that $M$ has no boundary. Then, if $K^{\mathrm{ext}}>0$, then we have
\begin{eqnarray*}
\#S^+-\#S^-=\#S^+_\star-\#S^-_\star.\label{5.18}
\end{eqnarray*}
On the other hand, if $K^{\mathrm{ext}}<0$, then we have
\begin{gather*}
2\left(\chi(M^+)-\chi(M^-)\right)+\#S^+-\#S^-=\#S^+_\star-\#S^-_\star,\nonumber\\
\chi(M)=0.\label{5.19}
\end{gather*}}
\end{corollary}

\end{document}